\documentclass{amsart}

\usepackage[latin1]{inputenc}
\usepackage{fancyhdr}
\usepackage{indentfirst}
\usepackage[dvips]{graphicx}
\usepackage{graphicx}
\usepackage{newlfont}
\usepackage{amssymb}
\usepackage{amsfonts}
\usepackage{amsmath}
\usepackage{latexsym}
\usepackage{amscd,amsthm}
\usepackage{psfrag}
\usepackage[usenames]{color}
\usepackage[all]{xy}

\newtheorem{theorem}{Theorem}[section]

\newtheorem{prop}[theorem]{Proposition}

\theoremstyle{definition}
\newtheorem{defn}[theorem]{Definition}

\theoremstyle{remark}
\newtheorem{rem}[theorem]{Remark}

\newcommand{\R}{\mathbb R}
\newcommand{\N}{\mathbb N}

\newcommand{\p}{\varphi}
\newcommand{\s}{\psi}

\newcommand{\fr}{\vec \varphi}

\newcommand{\rank}{{\mathrm{rank}}}
\newcommand{\newatop}[2]{\genfrac{}{}{0pt}{1}{#1}{#2}}
\newcommand{\im}{{\mathrm{Im}}}

\numberwithin{equation}{section}

\begin{document}

\title[Invariance properties of the multidimensional matching distance]{Invariance properties of the Multidimensional matching distance in Persistent Topology and Homology}
\author[A.Cerri]{Andrea Cerri}
\address{Andrea Cerri, ARCES, Universit\`a di Bologna,
via Toffano $2/2$, I-$40135$ Bologna, Italia\newline Dipartimento
di Matematica, Universit\`a di Bologna, P.zza di Porta S. Donato
5, I-$40126$ Bologna, Italia}
\email{cerri@dm.unibo.it}

\author[P. Frosini]{Patrizio Frosini}
\address{Patrizio Frosini, ARCES, Universit\`a di Bologna,
via Toffano $2/2$, I-$40135$ Bologna, Italia\newline Dipartimento
di Matematica, Universit\`a di Bologna, P.zza di Porta S. Donato
5, I-$40126$ Bologna, Italia, tel. +39-051-2094478, fax.
+39-051-2094490} \email{frosini@dm.unibo.it}



\date{\today} 



\begin{abstract}
Persistent Topology studies topological features of shapes
by analyzing the lower level sets of suitable functions, 
called filtering functions, and encoding the arising information 
in a parameterized version of the Betti numbers, i.e. the ranks
of persistent homology groups. Initially introduced by considering real-valued filtering functions, Persistent Topology has been subsequently generalized to a multidimensional setting, i.e. to the case of $\R^n$-valued filtering functions, leading to studying the ranks of multidimensional homology groups. 
In particular, a multidimensional matching distance has been defined, in order to compare these ranks. The definition of the multidimensional matching distance is based on foliating the domain of the ranks of multidimensional 
homology groups by a collection of half-planes, and hence it formally depends on a subset of $\R^n\times\R^n$ inducing a parameterization of these half-planes. It happens that it is possible to choose this subset in an infinite number of different ways. In this paper we show that the multidimensional matching distance is actually invariant with respect to such a choice. 
\end{abstract}

\maketitle

\section*{Introduction}
In the last two decades Persistent Topology has been introduced and studied to describe stable properties of sublevel sets of topological spaces endowed with real valued functions, representing shape or topological characteristics of real objects \cite{EdLeZo02,Fr91,Ro99}. In order to compare them, the concept of persistent homology group has been defined, and the scientific community has become more and more interested in this subject, both from the theoretical 
\cite{BiDeFaFrGiLaPaSp08,Ca09,EdHa08,Gh08} and the applicative point of view 
\cite{BiGiSpFa08,BuKi07,CaZoCoGu05,CeFeGi06,ChCoGuMeOu09,DeGh07,DiFrPa04,KaMiMr04,UrVe97}.

In recent years, this setting has been extended to manage properties that are described by $\mathbb{R}^n$-valued functions \cite{BiCeFrGiLa08,CaDiFe,CaZo09}, representing multidimensional 
properties (e.g., the color). Such an extension has revealed to be a spring of new interesting mathematical problems, and the search of their solutions has stimulated the introduction of new ideas. One of these ideas is represented by the so-called \emph{foliation method}, consisting in foliating the domain of the ranks of persistent homology groups by means of a family of half-planes \cite{BiCeFrGiLa08,CaDiFe,CeFr10}. An important consequence of this approach has been the recent introduction of a multidimensional matching distance between the ranks of persistent homology groups and the proof of its stability \cite{CeDiFeFrLa09}, opening the way to the application of multidimensional 
persistent homology in shape comparison.

However, the definition of the multidimensional matching distance formally depends on a subset of $\R^n\times\R^n$, inducing a parameterization of the half-planes in the considered collection. Since such a subset can be chosen in an infinite number of different ways, it follows that, in principle, each particular choice could lead to a different matching distance.
       
In this paper we solve this problem, proving that, in fact, the 
matching distance is independent of the chosen subset of parameters in $\R^n\times\R^n$. Beyond its own theoretical interest, this result allows us to change the parameterization in order to make our computations easier \cite{BiCeFrGi10}, without any change in our mathematical setting. 

\section{Preliminary definitions and results}
In this paper, each considered space is assumed to be
triangulable, i.e. there is a finite simplicial complex with
homeomorphic underlying space. In particular, triangulable spaces
are always compact and metrizable.

The following relations $\preceq$ and $\prec$ are defined in
$\R^n$: for $\vec u=(u_1,\dots,u_n)$ and $\vec v=(v_1,\dots,v_n)$,
we say $\vec u\preceq\vec v$ (resp. $\vec u\prec\vec v$) if and
only if $u_i\leq\ v_i$ (resp. $u_i<v_i$) for every index
$i=1,\dots,n$. Moreover, $\R^n$ is endowed with the usual
$\max$-norm: $\|(u_1,u_2,\dots,u_n)\|_{\infty}=\max_{1\leq i\leq
n}|u_i|$.

We shall use the following notations: $\Delta^+$ will be the open
set $\{\left(\vec u,\vec v\right)\in\R^n\times\R^n:\vec u\prec\vec v\}$.
Given a triangulable space $X$, for every $n$-tuple $\vec
u=(u_1,\dots,u_n)\in\R^n$ and for every function
$\vec\p:X\to\R^n$, we shall denote by $X\langle\fr\preceq \vec
u\,\rangle$ the set $\{x\in X:\varphi_i(x)\leq u_i,\
i=1,\dots,n\}$. Any function $\vec\varphi$ will said to be a 
{\em $n$-dimensional filtering (or measuring) function}.

The definition below extends the concept of the persistent
homology group to a multidimensional setting, i.e. to the case 
of filtering functions taking values in $\R^n$.

\begin{defn}
Let $\imath^{\left(\vec u,\vec v\right)}_k:\check{H}_k(X\langle\vec\p\preceq\vec
u\rangle)\rightarrow \check{H}_k(X\langle\vec\p\preceq\vec
v\rangle)$ be the homomorphism induced by the inclusion map
$\imath^{\left(\vec u,\vec v\right)}:X\langle\vec\p\preceq\vec
u\rangle\hookrightarrow X\langle\vec\p\preceq\vec v\rangle$ with
$\vec u\preceq\vec v$, where $\check{H}_k$ denotes the $k$th
\v{C}ech homology group. If $\vec u\prec\vec v$, the image of
$\imath^{\left(\vec u,\vec v\right)}_k$ is called the {\em multidimensional
$k$th persistent homology group of $(X,\vec\p)$ at $\left(\vec u, \vec
v\right)$}, and is denoted by $\check{H}_k^{\left(\vec u, \vec
v\right)}(X,\vec\p)$.
\end{defn}

In other words, the group $\check{H}_k^{\left(\vec u, \vec
v\right)}(X,\vec\p)$ contains all and only the $k$-homology classes of
cycles ``born'' before $\vec u$ and ``still alive'' at $\vec v$.

For details about \v{C}ech homology, the reader can refer to
\cite{EiSt}.

In what follows, we shall work with coefficients in a field
$\mathbb{K}$, so that homology groups are vector spaces, and hence
torsion-free. Therefore, they can be completely described by their
rank, leading to the following definition (cf. \cite{CaZo09,CeDiFeFrLa09}).

\begin{defn}[$k$th rank invariant]\label{Rank}
Let $X$ be a triangulable space, and $\vec\varphi:X\to\R^n$ a continuous
function. Let $k\in\mathbb{Z}$. The {\em $k$th rank invariant of
the pair $(X,\fr)$ over a field $\mathbb{K}$} is the function
$\rho_{(X,\fr),k}:\Delta^+\to\N$ defined as
$$
\rho_{(X,\vec\varphi),k}\left(\vec u,\vec v\right)=\rank \,\imath^{\left(\vec u,\vec v\right)}_k.
$$
\end{defn}

By the rank of a homomorphism we mean the dimension of its image. 
We observe that, in general, the rank invariant of a pair $(space, filtering\  function)$ can assume infinite value. On the other hand, in \cite{CeDiFeFrLa09} it has been proved that, under our assumptions on $X$ and $\vec\p$, the value $\infty$ is never attained by $\rho_{(X,\vec\varphi),k}$. Therefore, Definition \ref{Rank} is definitely well posed.

\subsection{The particular case $n=1$}\label{PartCase1}

Let us now analyze in a bit more detail the $1$-dimensional case, i.e. 
when the filtering function is real-valued.
Indeed, Persistent Topology has been widely
developed in this setting \cite{BiDeFaFrGiLaPaSp08}. 
For what concerns the particular case $k=0$, i.e. 
the case of homology of degree $0$, it was first introduced in 
the '90s under the name of {\em Size Theory} 
(see, e.g., \cite{Fr91,FrLa99,VeUrFrFe93}).

According to the terminology used in the literature about the case
$n=1$, the symbols $\vec\varphi$, $\vec u$, $\vec v$, $\prec$, $\preceq$ will be
replaced respectively by $\varphi$, $u$, $v$, $<$, $\leq$. 
Moreover, let us introduce some further notations: 
$\Delta=\partial\Delta^+$, $\Delta^*=\Delta^+\cup\{(u,\infty):u\in\R\}$, and
$\bar\Delta^*=\Delta^*\cup\Delta$. Finally, we write
$\|\p\|_\infty$ for $\max_{x\in X}|\p(x)|$.

When referring to a real valued filtering function 
$\varphi:X\to\R$, the rank invariant $\rho_{(X,\varphi),k}$
turns out to be a function defined over the open subset of the real plane
given by $\{(u,v)\in\R^2:u<v\}$, taking each point $(u,v)$ of the domain into the number 
of $k$-homology classes of cycles ``born'' before $u$ and ``still alive'' at $v$.
%

Figure \ref{FigEs} shows an example of a topological space $X$, 
endowed with a filtering function $\varphi:X\to\R$ (Figure \ref{FigEs}$(a)$),
together with the $0$th rank invariant $\rho_{(X,\varphi),0}$ (Figure \ref{FigEs}$(b)$). 
In this case, $X$ is the curve drawn by a solid line, and
$\varphi$ is the ordinate function.

\begin{figure}
\psfrag{M}{$X$} \psfrag{F}{\!$\varphi$}
\psfrag{N}{\!\!\!$\rho_{(X,\varphi),0}$}
\psfrag{(A)}{$(a)$}\psfrag{(E)}{$(b)$}
\psfrag{x}{$u$}\psfrag{y}{$v$}
\psfrag{a}{\textcolor[rgb]{0.00,0.00,1.00}{\!$c$}}
\psfrag{b}{\textcolor[rgb]{0.00,0.00,1.00}{\!$d$}}
\begin{center}
\includegraphics[height=4cm]{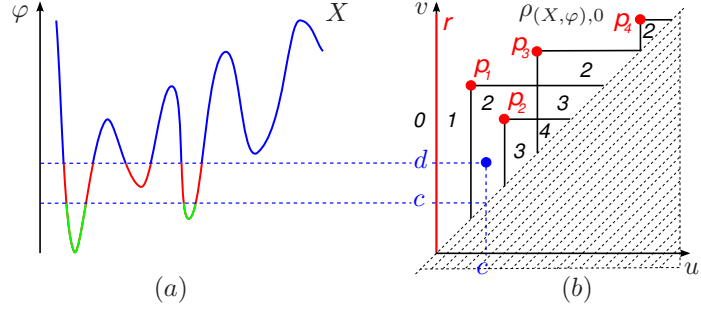}
\end{center}
\caption{$(a)$ The space $X$ and the
filtering function $\varphi$. $(b)$ The associated rank invariant
$\rho_{(X,\varphi),0}$.}\label{FigEs}
\end{figure}

As can be seen, the domain $\Delta^+=\{(u,v)\in\R^2:u<v\}$ is
divided into regions. Each one is labeled by a number, coinciding with
the constant value that $\rho_{(X,\varphi),0}$ takes in the interior of that region. 
For example, the value of $\rho_{(X,\varphi),0}$ at the point $(c,d)$ is equal to $2$. 

Due to its typical structure, it has been proved that the
information conveyed by a $1$-dimensional rank invariant can be
described in a very simple and compact way (cf. \cite{CeDiFeFrLa09,CoEdHa05,FrLa01,LaFr97}). 
More precisely, under the present assumption on $X$ and $\varphi$, 
and making use of \v{C}ech homology, in \cite{CeDiFeFrLa09} the authors show that 
each $1$-dimensional rank invariant can be compactly described by a multiset
of points, proper and at infinity, of the real plane, called respectively {\em proper cornerpoints} 
and {\em cornerpoints at infinity (or cornerlines)} and
defined as follows:

\begin{defn}[Proper cornerpoint]\label{Proper}
For every point $p=(u,v)\in\Delta^+$, we define the number
$\mu_k(p)$ as the minimum over all the positive real numbers
$\varepsilon$, with $u+\varepsilon<v-\varepsilon$, of
$$
\rho_{(X,\p),k}(u+\varepsilon,v-\varepsilon)-\rho_{(X,\p),k}(u-\varepsilon,v-\varepsilon)
-\rho_{(X,\p),k}(u+\varepsilon,v+\varepsilon)+\rho_{(X,\p),k}(u-\varepsilon,v+\varepsilon).
$$
The number $\mu_k(p)$ will be called the \emph{multiplicity} of
$p$ for $\rho_{(X,\p),k}$. Moreover, we shall call a {\em proper
cornerpoint for $\rho_{(X,\p),k}$} any point $p\in\Delta^+$ such
that the number $\mu_k(p)$ is strictly positive.
\end{defn}

\begin{defn}[Cornerpoint at infinity]\label{Cornerpoint}
For every vertical line $r$, with equation $u=\bar u$, $\bar
u\in\R$, let us identify $r$ with $(\bar u,\infty)\in\Delta^*$,
and define the number $\mu_k(r)$ as the minimum over all the
positive real numbers $\varepsilon$, with $\bar
u+\varepsilon<1/\varepsilon$, of
$$
\rho_{(X,\p),k}\left(\bar
u+\varepsilon,\frac{1}{\varepsilon}\right)-\rho_{(X,\p),k}\left(\bar
u-\varepsilon,\frac{1}{\varepsilon}\right).
$$
The number $\mu_k(r)$ will be called the \emph{multiplicity} of
$r$ for $\rho_{(X,\p),k}$. When this finite number is strictly
positive, we call $r$ a {\em cornerpoint at infinity for
$\rho_{(X,\p),k}$}.
\end{defn}

%
%

The concept of cornerpoint allows us to introduce a representation
of the rank invariant, based on the following definition
\cite{CeDiFeFrLa09,CoEdHa07}.

\begin{defn}[Persistence diagram]
The {\em persistence diagram} $D_k(X,\p)\subset\bar\Delta^*$ is
the multiset of all cornerpoints (both proper and at infinity) for
$\rho_{(X,\p),k}$, counted with their multiplicity, union the
points of $\Delta$, counted with infinite multiplicity.
\end{defn} 
 
The fundamental role of persistent diagrams is explicitly shown 
in the following Representation Theorem \ref{k-triangle} \cite{CeDiFeFrLa09,CoEdHa05}, 
claiming that they uniquely determine $1$-dimensional rank invariants (the
converse also holds by definition of persistence diagram). 

\begin{theorem}[Representation Theorem]\label{k-triangle}
For every $(\bar u,\bar v)\in\Delta^+$, we have
$$
\rho_{(X,\p),k}(\bar u,\bar v)=\sum_{\newatop{(u,v)\in\Delta^*}{ 
u\leq\bar u,\,v>\bar v}}\mu_k((u,v)).
$$
\end{theorem}

Roughly speaking, the Representation Theorem
\ref{k-triangle} claims that the value assumed by 
$\rho_{(X,\p),k}$ at a point $(\bar u,\bar v)\in\Delta^+$ equals the number of
cornerpoints lying above and on the left of $(\bar u,\bar v)$. By
means of this theorem we are able to compactly represent
$1$-dimensional rank invariants as multisets of cornerpoints
and cornerpoints at infinity, i.e. as persistent diagrams. For example,
the $1$-dimensional rank invariant shown in Figure \ref{FigEs}$(b)$ admits
persistent diagram associated to the multiset given by $r,p_1,p_2,p_3,p_4$, 
where $r$ is the only cornerpoint at infinity,
with coordinates $(0,\infty)$, and each element has multiplicity equal to $1$.

As a consequence of the Representation Theorem \ref{k-triangle}
any distance between persistence diagrams induces a distance
between one-dimensional rank invariants. This justifies the
following definition \cite{CeDiFeFrLa09,CoEdHa07,dAFrLa}. 

\begin{defn}[Matching distance]\label{MatchingDistance}
Let $X$ be a triangulable space endowed with continuous functions
$\p,\s:X\to\R$. The {\em matching distance} $d_{match}$ between
$\rho_{(X,\p),k}$ and $\rho_{(X,\s),k}$ is defined to be the bottleneck
distance between $D_k(X,\p)$ and $D_k(X,\s)$, i.e.
\begin{eqnarray}\label{DistMatch}
d_{match}\left(\rho_{(X,\p),k},\rho_{(X,\s),k}\right)=\inf_{\gamma}\max_{p\in
D_k(X,\p)}\|p-\gamma(p)\|_{\widetilde{\infty}},
\end{eqnarray}
where $\gamma$ ranges over all multi-bijections (i.e. bijections between multisets) between
$D_k(X,\p)$ and $D_k(X,\s)$, and for every $p=(u,v),q=(u',v')$ in
$\Delta^*$,
$$
\|p-q\|_{\widetilde{\infty}}=
\min\left\{\max\left\{|u-u'|,|v-v'|\right\},\max\left\{\frac{v-u}{2},\frac{v'-u'}{2}\right\}\right\},
$$
with the convention about points at infinity that
$\infty-y=y-\infty=\infty$ when $y\neq\infty$, $\infty-\infty=0$,
$\frac{\infty}{2}=\infty$, $|\infty|=\infty$, $\min\{c,\infty\}=c$
and $\max\{c,\infty\}=\infty$.
\end{defn}

In plain words, $\|\cdot\|_{\widetilde{\infty}}$ measures the
pseudodistance between two points $p$ and $q$ as the minimum
between the cost of moving one point onto the other and the cost
of moving both points onto the diagonal, with respect to the
max-norm and under the assumption that any two points of the
diagonal have vanishing pseudodistance (we recall that a pseudodistance 
$d$ is just a distance missing the condition $d(X,Y)=0\Rightarrow X=Y$, 
i.e. two distinct elements may have vanishing distance with respect to $d$).
 
An application of the matching distance is given by Figure
\ref{DistMatch22}$(c)$. 

\begin{figure}[ht]
\psfrag{(A)}{$(a)$}\psfrag{(B)}{$(b)$}\psfrag{(C)}{$(c)$}
\psfrag{x}{$u$}\psfrag{y}{\!$v$}\psfrag{optimal}{\!\!\!\!{\footnotesize matching}}\psfrag{matching}{}
\begin{center}
\includegraphics[width=0.8\textwidth]{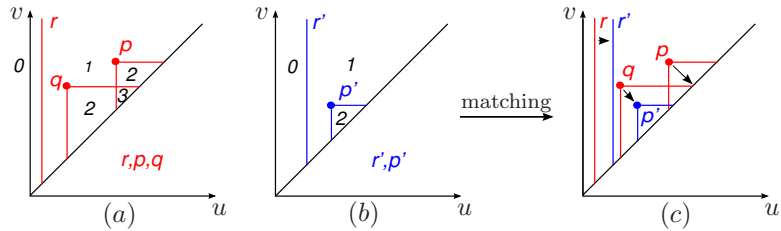}
\end{center}
\caption{$(a)$ The rank invariant corresponding to the persistent 
diagram given by $r,p,q$. $(b)$ The rank invariant corresponding 
to the persistent diagram given by $r',p'$. $(c)$ The matching between 
the two persistent diagrams, realizing the matching distance between 
the two rank invariants.}\label{DistMatch22}
\end{figure}

As can be seen by this example, different
$1$-dimensional rank invariants may in general have a different
number of cornerpoints. Therefore $d_{match}$ allows a proper
cornerpoint to be matched to a point of the diagonal: this
matching can be interpreted as the destruction of a proper
cornerpoint.  Furthermore, we stress that the matching distance is
stable with respect to perturbations of the filtering functions,
as the following Matching Stability Theorem states:

\begin{theorem}[One-Dimensional Stability Theorem]\label{StabilityTheorem}
Assume that $X$ is a triangulable space, and $\p,\s:X\to\R$ are two continuous
functions. Then it holds that
$d_{match}(\rho_{(X,\p),k},\rho_{(X,\s),k})\leq\|\p-\s\|_\infty$.
\end{theorem}

For a proof of the previous theorem and more details about the
matching distance the reader is referred to \cite{CeDiFeFrLa09,dAFrLabis,dAFrLa} 
(see also \cite{ChCoGlGuOu09,CoEdHa05} for the bottleneck distance). 

\subsection{The {\boldmath $n$}-dimensional case}\label{nDim}

Let us go back to consider a filtering function $\vec\varphi:X\to\R^n$.
In \cite{BiCeFrGiLa08,CaDiFe}, the authors show that the case $n>1$ 
can be reduced to the $1$-dimensional setting by a change of 
variable and the use of a suitable foliation of $\Delta^+$,
as can be seen in what follows. We shall refer to such a procedure as to the \emph{foliation method}.

Let us start by recalling that the following  parameterized family of
half-planes in $\R^n\times\R^n$ is a foliation of $\Delta ^+$. 

\begin{defn}\label{FoliationDef} 
For every vector
$\vec{l}=(l_1,\ldots,l_n)$ in $\mathbb{R}^n$
with $l_i>0$ for $i=1,\dots,n$ and $\sum_{i=1}^n l_i^2=1$, and for every vector $\vec{b}=(b_1,\ldots,b_n)$ in
$\mathbb{R}^n$, such that $\sum_{i=1}^n b_i=0$, the pair 
$\left(\vec{l},\vec{b}\right)$ will be said {\em admissible}. 
We shall denote by $Adm_n$ the set of all admissible pairs in 
$\R^n\times\R^n$. For every admissible pair, the half-plane  
$\pi_{\left(\vec{l},\vec{b}\right)}\subseteq\Delta^+$ is defined 
by the parametric equations
$$
\left\{%
\begin{array}{ll}
    \vec u=s\vec l + \vec b\\
    \vec v=t\vec l + \vec b\\
\end{array}%
\right.,
$$
with $s,t\in\R$ and $s<t$.
\end{defn}

In what follows, we shall use the symbol $\pi_{\left(\vec l,\vec b\right)}$ when referring to such a half-plane as a set of points, and the symbol 
$\pi_{\left(\vec l,\vec b\right)}:\vec u=s\vec l + \vec b,\ \vec v=t\vec l + \vec b$ when referring to its parameterization.

\begin{rem}\label{FoliationRem}
It can be verified that the half-planes collection $\left\{\pi_{\left(\vec l,\vec b\right)}:\vec u=s\vec l+\vec b,\right.$ $\left.v=t\vec l+\vec b\ |\  s<t\right\}_{\left(\vec l,\vec b\right)\in Adm_n}$ is actually a foliation of $\Delta^+$. Therefore, it follows that for every $\left(\vec u, \vec v\right)\in\Delta^+$, there exists one and only one admissible pair $\left(\vec l,\vec b\right)\in Adm_n$ such that $\left(\vec u, \vec v\right)\in\pi_{\left(\vec l,\vec b\right)}$. Moreover, for every $\left(\vec l,\vec b\right)\in Adm_n$ the half-plane $\pi_{\left(\vec l,\vec b\right)}$ is a subset of $\Delta^+$\cite{BiCeFrGiLa08}.
\end{rem}

The key property of the foliation defined in Definition \ref{FoliationDef} 
is  that the restriction of $\rho_{(X,\fr),k}$ to each leaf can be seen as a particular
$1$-dimensional rank invariant, as the following theorem states.

\begin{theorem}[Reduction Theorem]\label{reduction}
Let $\left(\vec{l},\vec{b}\right)$ be an admissible pair, and $F_{\left(\vec
l,\vec b\right)}^{\fr}:X\rightarrow\R$ be defined by setting
$$
F_{\left(\vec l,\vec
b\right)}^{\fr}(x)=\max_{i=1,\dots,n}\left\{\frac{\varphi_i(x)-b_i}{l_i}\right\}\
.
$$
Then, for every $\left(\vec u,\vec v\right)=\left(s\vec l+\vec b,t\vec l + \vec
b\right)\in\pi_{\left(\vec{l},\vec{b}\right)}$ the following equality holds:
$$
\rho_{(X,\fr),k}\left(\vec u,\vec v\right)=\rho_{(X,F_{\left(\vec l,\vec
b\right)}^{\fr}),k}(s,t)\ .
$$
\end{theorem}

In the following, we shall use the symbol $\rho_{(X,F_{\left(\vec l,\vec
b\right)}^{\fr}),k}$ in the sense of the Reduction Theorem \ref{reduction}.

As a consequence of the Reduction Theorem \ref{reduction}, we
observe that the identity
$\rho_{(X,\fr),k}\equiv\rho_{(X,\vec{\psi}),k}$ holds if and only
if $d_{match}(\rho_{(X,F_{\left(\vec l,\vec
b\right)}^{\fr}),k},\rho_{(X,F_{\left(\vec l,\vec b\right)}^{\vec\psi}),k})=0$, for
every admissible pair $\left(\vec{l},\vec{b}\right)$.

Furthermore, the Reduction Theorem \ref{reduction} allows us to represent a
multidimensional rank invariant $\rho_{(X,\fr),k}$ by a
collection of persistent diagrams, following the
machinery described in Subsection \ref{PartCase1} for the case $n=1$.
Indeed, each admissible pair $\left(\vec l,\vec b\right)$ can be associated
with a persistent diagram $D_k(X,F_{\left(\vec l,\vec
b\right)}^{\fr})$ describing the
$1$-dimensional rank invariant $\rho_{(X,F_{\left(\vec l,\vec
b\right)}^{\fr}),k}$. Therefore, for every half-plane in 
$\left\{\pi_{\left(\vec l,\vec b\right)}:\vec u=s\vec l+\vec b,\ \vec v=t\vec l+\vec b\ |\  s<t\right\}_{\left(\vec l,\vec b\right)\in Adm_n}$ the matching distance between $1$-dimensional rank invariants
can be applied, 
leading to the following definition of a proven stable
distance between two multidimensional rank invariants \cite{BiCeFrGiLa08,CeDiFeFrLa09}:

\begin{defn}\label{MultiDimDist}
Let $X$, $Y$ be two triangulable spaces, and let $\vec\varphi:X\to\R^n$, $\vec\psi:Y\to\R^n$ be two filtering functions. 
The {\em multidimensional matching distance} $D_{match}(\rho_{(X,\fr),k},\rho_{(Y,\vec\psi),k})$ 
is the (extended) distance defined by setting 
$$
D_{match}(\rho_{(X,\fr),k},\rho_{(Y,\vec\psi),k})=\sup_{\left(\vec
l,\vec b\right)\in Adm_n}\min_i l_i\cdot d_{match}(\rho_{(X,F_{\left(\vec
l,\vec b\right)}^{\fr}),k},\rho_{(Y,F_{\left(\vec l,\vec
b\right)}^{\vec\psi}),k})$$
\end{defn}

\begin{rem}\label{RemPseudoDist}
The term ``extended'' in Definition \ref{MultiDimDist} refers to the fact that, 
if the spaces $X$ and $Y$ are not assumed to be homotopically equivalent, 
the multidimensional matching distance $D_{match}$ still verifies
all the properties of a distance, except for the fact that it may
take the value $\infty$.

\end{rem}

\section{New results}

The above multidimensional Persistent Topology framework leads us to the following considerations. As can be seen in Definition \ref{MultiDimDist}, the concept of $D_{match}$ depends on the half-planes foliating $\Delta^+$, with particular reference to the set $Adm_n$ of admissible pairs. On the other hand, it is worth nothing that the choice of $Adm_n$ is completely arbitrary. 
Indeed, the machinery recalled in Section \ref{nDim} could be applied by taking into account any other set of parameters $\Lambda\times B\in\R^n\times\R^n$ such that the half-planes collection $\left\{\pi_{\left(\vec\lambda,\vec\beta\right)}:\vec u=s\vec\lambda+\vec\beta,\,\vec v=t\vec\lambda+\vec\beta\ |\  s<t\right\}_{\left(\vec\lambda,\vec\beta\right)\in\Lambda\times B}$ satisfies a property analogous to the one described in Remark \ref{FoliationRem}. More precisely, it is sufficient that 
for every $\left(\vec u, \vec v\right)\in\Delta^+$, there exists one and only one $\left(\vec\lambda,\vec\beta\right)\in\Lambda\times B$ with
$\left(\vec u, \vec v\right)\in\pi_{\left(\vec\lambda,\vec\beta\right)}$, and that for every $\left(\vec\lambda,\vec\beta\right)\in\Lambda\times B$, the half-plane $\pi_{\left(\vec\lambda,\vec\beta\right)}$ is a subset of $\Delta^+$.

In order to clarify our last assertion, let us consider, e.g., the set $Ladm_n\in\R^n\times\R^n$ containing the pairs $\left(\vec\lambda,\vec\beta\right)$ with $\vec\lambda=(\lambda_1,\dots,\lambda_n)$ such that $\sum_{i=1}^n\lambda_i=1$ and $\lambda_i>0$ for $i=1,\dots,n$, and $\vec\beta=(\beta_1,\dots,\beta_n)$ with $\sum_{i=1}^n\beta_i=0$. It can be shown that, in this case, for every $\left(\vec u, \vec v\right)\in\Delta^+$ there exists one and only one $\left(\vec\lambda,\vec\beta\right)\in Ladm_n$ such that $\left(\vec u, \vec v\right)\in\pi_{\left(\vec\lambda,\vec\beta\right)}$. To this aim, let us set $\lambda_i=\frac{v_i-u_i}{\sum_{j=1}^n(v_j-u_j)}$ and $ \beta_i=\frac{u_i\sum_{j=1}^n v_j-v_i\sum_{j=1}^n u_j}{\sum_{j=1}^n (v_j-u_j)}$ for every $i=1,\dots,n$. Obviously, for every $\left(\vec\lambda,\vec\beta\right)\in Ladm_n$ we have also that the half-plane $\pi_{\left(\vec\lambda,\vec\beta\right)}$ is a subset of $\Delta^+$.

As a consequence, an analogue of the Reduction Theorem \ref{reduction} can be proved, leading to a similar, but {\em formally different} version of multidimensional matching distance between rank invariants. More precisely, under the same hypotheses assumed in Definition \ref{MultiDimDist} for the spaces $X$, $Y$ and the filtering functions $\vec\varphi$,$\vec\psi$, we can define the distance $\widetilde D_{match}(\rho_{(X,\fr),k},\rho_{(Y,\vec\psi),k})$ 
by setting 
{\setlength\arraycolsep{2pt}
\begin{eqnarray}\label{MultiDimDistEq}
\ \widetilde D_{match}(\rho_{(X,\vec\varphi),k}&,&\rho_{(Y,\vec\psi),k})=\\
&&=\sup_{\left(\vec\lambda,\vec\beta\right)\in Ladm_n}\min_i\lambda_i\cdot d_{match}(\rho_{(X,F_{\left(\vec
\lambda,\vec\beta\right)}^{\fr}),k},\rho_{(Y,F_{\left(\vec\lambda,\vec\beta\right)}^{\vec\psi}),k}),\nonumber
\end{eqnarray}}
with $F_{\left(\vec\lambda,\vec
\beta\right)}^{\vec\varphi}=\max_{i=1,\dots,n}\left\{\frac{\varphi_i-\beta_i}{\lambda_i}\right\}$ and $F_{\left(\vec\lambda,\vec
\beta\right)}^{\vec\psi}=\max_{i=1,\dots,n}\left\{\frac{\psi_i-\beta_i}{\lambda_i}\right\}$.

Following these reasonings, we are quite naturally induced to wonder if $D_{match}$ and $\widetilde D_{match}$ are effectively different distances between multidimensional rank invariants. It is possible to prove that the answer to such a question is negative, that is, $D_{match}$ and $\widetilde D_{match}$ coincide. To see this, we can define a bijection between $Adm_n$ and $Ladm_n$, taking each $\left(\vec l,\vec b\right)\in Adm_n$ to the unique $\left(\vec\lambda,\vec\beta\right)\in Ladm_n$ with $\vec l=c\vec\lambda$ ($c\neq 0$), $\vec b=\vec\beta$, and prove that $\min_i l_i\cdot d_{match}(\rho_{(X,F_{\left(\vec
l,\vec b\right)}^{\fr}),k},\rho_{(Y,F_{\left(\vec l,\vec b\right)}^{\vec\psi}),k})=\min_i\lambda_i\cdot d_{match}(\rho_{(X,F_{\left(\vec
\lambda,\vec\beta\right)}^{\fr}),k},\rho_{(Y,F_{\left(\vec\lambda,\vec\beta\right)}^{\vec\psi}),k})$, with the last equality coming from a property of $d_{match}$ we shall formally prove in Proposition \ref{Prop3}.

Before going on, let us remark that the coincidence between $D_{match}$ and $\widetilde D_{match}$ has revealed to be useful in simplifying some technical details in a recent work concerning the effective computation of the multidimensional matching distance \cite{BiCeFrGi10}. Indeed, it allows us to substitute $Adm_n$ with $Ladm_n$ in order to make our computations easier, without any modification in our mathematical setting.

In the light of the previous example, our goal in what follows is to show that the same considerations hold for any ``admissible'' choice of the set of parameters $\Lambda\times B\in\R^n\times\R^n$. To be more precise, we shall prove that, if the half-planes collection $\left\{\pi_{\left(\vec\lambda,\vec\beta\right)}:\vec u=s\vec\lambda+\vec\beta,\ \vec v=t\vec\lambda+\vec\beta\ |\  s<t\right\}_{\left(\vec\lambda,\vec\beta\right)\in\Lambda\times B}$ actually foliates $\Delta^+$, then the induced matching distance (in the sense of the analogue of equation (\ref{MultiDimDistEq})) between multidimensional rank invariants always coincides with $D_{match}$ (Theorem \ref{InsightTotal}).

\subsection{1-dimensional rank invariants and monotonic changes of the associated filtering functions}
In order to provide the main theorem of this paper, let us first show some new results about the changes of $1$-dimensional rank invariants with respect to the composition of the associated filtering functions with a strictly increasing map (Proposition \ref{Prop1} and Proposition \ref{Prop2}). With particular reference to the case of composition with strictly increasing affine maps, we will show how these results affect 
the $1$-dimensional matching distance (Proposition \ref{Prop3}). 

\begin{prop}\label{Prop1}
Assume that $f:\R\to\R$ is a strictly increasing function. Then it follows that 
$\rho_{(X,\p),k}(u,v)=\rho_{(X,f\circ\p),k}(f(u),f(v))$, for every $(u,v)\in\Delta^+$.
\end{prop}
\begin{proof}
Since $f:\R\to\R$ is a strictly increasing function, if $(u,v)\in\Delta^+$ then $(f(u),f(v))\in\Delta^+$. 
The claim easily follows by observing that 
$X\langle\varphi\preceq u\rangle=X\langle f\circ\varphi\preceq f(u)\rangle$, for every $u\in\R$.
%
%
\end{proof}

As a consequence of Proposition \ref{Prop1}, we have the following result (we skip the easy proof), 
stating that the composition of the considered filtering function with strictly increasing maps preserves 
the multiplicity of cornerpoints in the associated $1$-dimensional rank invariant.  
\begin{prop}\label{Prop2}
Assume that $f:\R\to\R$ is a strictly increasing function. 
Then, for every $(u,v)\in\Delta^+$ and $(\bar u,\infty)\in\Delta^*$, 
it holds that $\mu_k((u,v))=\mu_k((f(u),f(v)))$ and $\mu_k((\bar u,\infty))
=\mu_k((f(\bar u),\infty))$, respectively.
\end{prop}

Let us now confine ourselves to the assumption that $f:\R\to\R$ is defined as $f(x)=ax+b$, 
with $a,b\in\R$ and $a>0$. In this case, from Definition \ref{MatchingDistance} and by applying 
Proposition \ref{Prop2} we obtain the next result about the matching distance 
between $1$-dimensional rank invariants.

\begin{prop}\label{Prop3}
Assume that $f:\R\to\R$ is defined as $f(x)=ax+b$, with $a,b\in\R$ and $a>0$. Let also $\p:X\to\R$, $\s:Y\to\R$ 
be two filtering functions for the triangulable spaces $X$ and $Y$, respectively. 
Then, it holds that
$$
d_{match}\left(\rho_{(X,f\circ\p),k},\rho_{(Y,f\circ\s),k}\right)=a\cdot d_{match}\left(\rho_{(X,\p),k},\rho_{(Y,\s),k}\right).
$$
\end{prop}
\begin{proof}
Assume that $(u,v),(u',v')\in\Delta^*$, with $(u,v)\in D_k(X,\varphi)$ and $(u',v')\in D_k(Y,\psi)$. Then, recalling the assumptions on the function $f$, it follows that $(f(u),f(v)),(f(u'),f(v'))\in\Delta^*$, with the convention about points at infinity that $f(\infty)=\infty$. Moreover, following the definition of the operator $\|\cdot\|_{\widetilde{\infty}}$ 
in Definition \ref{MatchingDistance} we have
{\setlength\arraycolsep{2pt}
\begin{eqnarray}
&&\|(f(u),f(v))\!-\!(f(u'),f(v'))\|_{\widetilde{\infty}}=\nonumber\\
&=&\!\min\!\left\{\max\left\{|f(u)\!-\!f(u')|,|f(v)\!-\!f(v')|\right\},\max\!\left\{\frac{f(v)\!-\!f(u)}{2},\frac{f(v')\!-\!f(u')}{2}\right\}\!\right\}\!=\nonumber\\
&=&\!\min\!\left\{\max\left\{a\cdot|u\!-\!u'|,a\cdot|v\!-\!v'|\right\},\max\!\left\{a\cdot\frac{v\!-\!u}{2},a\cdot\frac{v'\!-\!u'}{2}\right\}\!\right\}=\!\nonumber\\
&=&a\cdot\|(u,v)\!-\!(u',v')\|_{\widetilde{\infty}}.\nonumber
\end{eqnarray}}

Thus the claim follows from the definition of $d_{match}$ (Definition \ref{MatchingDistance}), from Proposition \ref{Prop2} and by observing that the correspondence taking each pair $(u,v)\in\Delta^*$ to the pair $(f(u),f(v))\in\Delta^*$ is actually a bijection.
\end{proof}

\subsection{Main results}
Let us go back to the main goal of this work. In what follows, we suppose that two sets $\Lambda,B\subseteq\R^n$ are given, such that the half-planes collection $\left\{\pi_{\left(\vec\lambda,\vec\beta\right)}:\vec u=s\vec\lambda+\vec\beta,\ \vec v=t\vec\lambda+\vec\beta\ |\  s<t\right\}_{\left(\vec\lambda,\vec\beta\right)\in\Lambda\times B}$ satisfies a property analogous to the one described in Remark \ref{FoliationRem}. More precisely, we are assuming that for every $\left(\vec u, \vec v\right)\in\Delta^+$ there exists one and only one $\left(\vec\lambda,\vec\beta\right)\in\Lambda\times B$ with $\left(\vec u, \vec v\right)\in\pi_{\left(\vec\lambda,\vec\beta\right)}$, and that for every $\left(\vec\lambda,\vec\beta\right)\in\Lambda\times B$, the half-plane $\pi_{\left(\vec\lambda,\vec\beta\right)}$ is a subset of $\Delta^+$. We observe that an infinite number of possibilities are available in order to choose the set $\Lambda\times B$. As a simple example, we could consider any variation of $Adm_n$ obtained by substituting the condition $\sum_{i=1}^n l_i^2=1$ in Definition \ref{FoliationDef} with $\sum_{i=1}^n l_i^p=1$, for any positive natural number $p$. Let also $X,Y$ be two triangulable spaces, and $\vec\varphi:X\to\R^n$, $\vec\psi:Y\to\R^n$ two filtering functions. Setting $F_{\left(\vec\lambda,\vec
\beta\right)}^{\vec\varphi}=\max_{i=1,\dots,n}\left\{\frac{\varphi_i-\beta_i}{\lambda_i}\right\}$ and $F_{\left(\vec\lambda,\vec
\beta\right)}^{\vec\psi}=\max_{i=1,\dots,n}\left\{\frac{\psi_i-\beta_i}{\lambda_i}\right\}$, we want to prove that the distance $\widehat D_{match}$ between multidimensional rank invariants, defined as 
$\widehat D_{match}(\rho_{(X,\vec\varphi),k},\rho_{(Y,\vec\psi),k})=\sup_{\left(\vec\lambda,\vec\beta\right)\in\Lambda\times B}\min_i\lambda_i\cdot d_{match}(\rho_{(X,F_{\left(\vec\lambda,\vec\beta\right)}^{\fr}),k},\rho_{(Y,F_{\left(\vec\lambda,\vec\beta\right)}^{\vec\psi}),k})$, 
coincides with the multidimensional matching distance $D_{match}$ introduced in Definition \ref{MultiDimDist}. 

The following two propositions give insights on the elements of the set $\Lambda$.

\begin{prop}\label{Insight1}
For every $\vec\lambda=(\lambda_1,\dots,\lambda_n)\in\Lambda$, it holds that $\lambda_i>0$, for $i=1,\dots,n$.
\end{prop}
\begin{proof}
For every $\vec\lambda\in\Lambda$, with $\vec\lambda=(\lambda_1,\dots,\lambda_n)$, we can arbitrarily choose $\vec\beta=(\beta_1,\dots,\beta_n)\in B$, and fix a point $\left(\vec u,\vec v\right)\in\pi_{\left(\vec\lambda,\vec\beta\right)}$, with $\vec u=(u_1,\dots,u_n)$ and
$\vec v=(v_1,\dots,v_n)$. Then two values $s,t\in\R$ exist, with $s<t$ and such that
$$
\left\{%
\begin{array}{ll}
    \vec u=s\vec\lambda + \vec\beta\\
    \vec v=t\vec\lambda + \vec\beta\\
\end{array}%
\right..
$$
The claim remains proved by observing that, for every index $i=1,\dots,n$, we have 
$0<v_i-u_i=(t-s)\cdot \lambda_i$ and hence, since $s<t$, $\lambda_i=\frac{v_i-u_i}{t-s}>0$.
\end{proof}

\begin{prop}\label{Insight2}
For every $\vec w=(w_1,\dots,w_n)\in\R^n$ with $w_i>0$ for
$i=1,\dots,n$, there exists one and only one $\vec\lambda\in\Lambda$ such that 
$\vec w=a\cdot\vec\lambda$ for a suitable real value $a>0$.
\end{prop}
\begin{proof}
For every $\vec w=(w_1,\dots,w_n)\in\R^n$ with 
$w_i>0$ for $i=1,\dots,n$, the pair $({\mathbf 0},\vec w)$ belongs to $\Delta^+$, ${\mathbf 0}$ representing the null vector of $\R^n$. 
From the assumptions on the collection $\left\{\pi_{\left(\vec\lambda,\vec\beta\right)}:\vec u=s\vec\lambda+\vec\beta,\ \vec v=t\vec\lambda+\vec\beta\ |\  s<t\right\}_{\left(\vec\lambda,\vec\beta\right)\in\Lambda\times B}$, it follows that there exists $\left(\vec\lambda,\vec\beta\right)\in\Lambda\times B$ such that $({\mathbf 0},\vec w)\in\pi_{\left(\vec\lambda,\vec\beta\right)}$, that is, ${\mathbf 0}=s\vec\lambda+\vec\beta,\vec w=t\vec\lambda+\vec\beta$ and $\vec w=(t-s)\vec\lambda$ for a suitable pair $(s,t)$, with $s<t$. Therefore, for every index $i=1,\dots,n$ it holds that $w_i=(t-s)\cdot\lambda_i>0$, thus proving (setting $a=t-s$) the existence of a vector $\vec\lambda\in\Lambda$ satisfying the claim. In order to show that such a $\vec\lambda$ is unique, let us suppose the existence of a vector $\vec\lambda'\in\Lambda$ with $\vec\lambda'\neq\vec\lambda$ and $\vec w=a'\cdot\vec\lambda'$ for a suitable real value $a'>0$. In this case, we would have $\vec w=a\cdot\vec\lambda=a'\cdot\vec\lambda'$ and hence $\vec\lambda=\frac{a'}{a}\cdot\vec\lambda'$. Therefore, this last equality would imply
$$
\left\{%
\begin{array}{ll}
    {\mathbf 0}=s\vec\lambda + \vec\beta=s\cdot\frac{a'}{a}\cdot\vec\lambda'+\vec\beta=s'\vec\lambda'+\vec\beta\\
    \vec w=t\vec\lambda + \vec\beta=t\cdot\frac{a'}{a}\cdot\vec\lambda'+\vec\beta=t'\vec\lambda'+\vec\beta\\
\end{array}%
\right.,
$$
where $s'=s\cdot\frac{a'}{a}$ and $t'=t\cdot\frac{a'}{a}$. In other words, the point $({\mathbf 0},\vec w)\in\Delta^+$ 
would lie on two half-planes associated with two different pairs $\left(\vec\lambda,\vec\beta\right),\left(\vec\lambda',\vec\beta\right)\in\Lambda\times B$.
 This contradiction concludes the proof.
\end{proof}

In what follows, for each considered $\left(\vec \lambda,\vec \beta\right)\in\R^n\times\R^n$, we shall use the symbol $\pi_{\left(\vec\lambda,\vec\beta\right)}$ to denote the half-plane 
$$
\left\{%
\begin{array}{ll}
    \vec u=s\vec\lambda + \vec\beta\\
    \vec v=t\vec\lambda + \vec\beta\\
\end{array}%
\right.,
$$
with $s<t$, and the symbols $F_{\left(\vec\lambda,\vec\beta\right)}^{\vec\varphi}$, $F_{\left(\vec\lambda,\vec\beta\right)}^{\vec\psi}$ in the sense of the Reduction Theorem \ref{reduction}, that is, $F_{\left(\vec\lambda,\vec\beta\right)}^{\vec\varphi}=\max_{i=1,\dots,n}\left\{\frac{\varphi_i-\beta_i}{\lambda_i}\right\}$ and $F_{\left(\vec\lambda,\vec\beta\right)}^{\vec\psi}=\max_{i=1,\dots,n}\left\{\frac{\psi_i-\beta_i}{\lambda_i}\right\}$.
Moreover, we shall write $d_{\left(\vec\lambda,\vec\beta\right)}(\rho_{(X,\vec\varphi),k},\rho_{(Y,\vec\psi),k})$ to denote the value $\min_{i=1,\dots,n}\lambda_i\cdot d_{match}(\rho_{(X,F_{\left(\vec\lambda,\vec\beta\right)}^{\vec\varphi}),k},\rho_{(Y,F_{\left(\vec\lambda,\vec\beta\right)}^{\vec\psi}),k})$. 
Finally, for every $\vec w=(w_1,\dots,w_n)\in\R^n$, the Euclidean norm $\sqrt{w_1^2+\dots+w_n^2}$ will be denoted by the symbol $\left\|\vec w\right\|$.

The next result allows us to assume that each vector $\vec\lambda=(\lambda_1,\dots,\lambda_n)\in\Lambda$ is a unit vector (with respect to the Euclidean norm). 

\begin{prop}\label{Insight3}
Let $\Lambda^*$ be the set containing all and only the vectors $\vec\lambda^*=(\lambda_1^*,\dots,\lambda_n^*)\in\R^n$ with $\lambda_i^*>0$ for every $i=1,\dots,n$, and $\left\|\vec\lambda^*\right\|=1$. Then it holds that $\underset{\left(\vec\lambda,\vec\beta\right)\in\Lambda\times B}{\sup} d_{\left(\vec\lambda,\vec\beta\right)}(\rho_{(X,\vec\varphi),k},\rho_{(Y,\vec\psi),k})=\underset{\left(\vec\lambda^*,\vec\beta\right)\in\Lambda^*\times B}{\sup} d_{\left(\vec\lambda^*,\vec\beta\right)}(\rho_{(X,\vec\varphi),k},\rho_{(Y,\vec\psi),k})$.
\end{prop}
\begin{proof}
From Proposition \ref{Insight1} and Proposition \ref{Insight2} it follows that a bijection between $\Lambda\times B$ and $\Lambda^*\times B$ exists, taking each $\left(\vec\lambda,\vec\beta\right)\in\Lambda\times B$ to $\left(\frac{\vec\lambda}{\left\|\vec\lambda\right\|},\vec\beta\right)\in\Lambda^*\times B$. 
Let us now fix $\vec\lambda\in\Lambda$, and set $\vec\lambda^*=\frac{\vec\lambda}{\left\|\vec\lambda\right\|}$.
We observe that the equalities $F^{\vec\varphi}_{\left(\vec\lambda^*,\vec\beta\right)}(x)=\left\|\vec\lambda\right\|F^{\vec\varphi}_{\left(\vec\lambda,\vec\beta\right)}(x)$ and $F^{\vec\psi}_{\left(\vec\lambda^*,\vec\beta\right)}(y)=\left\|\vec\lambda\right\|F^{\vec\psi}_{\left(\vec\lambda,\vec\beta\right)}(y)$ hold for every $x\in X$ and for every $y\in Y$, respectively. Hence, by Proposition \ref{Prop3} we have  $d_{match}(\rho_{(X,F^{\vec\varphi}_{\left(\vec\lambda^*,\vec\beta\right)}),k},\rho_{(Y,F^{\vec\psi}_{\left(\vec\lambda^*,\vec\beta\right)}),k})=\left\|\vec\lambda\right\|d_{match}(\rho_{(X,F^{\vec\varphi}_{\left(\vec\lambda,\vec\beta\right)}),k},\rho_{(Y,F^{\vec\psi}_{\left(\vec\lambda,\vec\beta\right)}),k})$, leading to  
{\setlength\arraycolsep{2pt}
\begin{eqnarray*}
&\ &d_{\left(\vec\lambda^*,\vec\beta\right)}(\rho_{(X,\vec\varphi),k},\rho_{(Y,\vec\psi),k})=\min_{i=1,\dots,n}\lambda^*_i\cdot d_{match}(\rho_{(X,F^{\vec\varphi}_{\left(\vec\lambda^*,\vec\beta\right)}),k},\rho_{(Y,F^{\vec\psi}_{\left(\vec\lambda^*,\vec\beta\right)}),k})=\nonumber\\
&=&\frac{\min_i \lambda_i}{\left\|\vec\lambda\right\|}\cdot\left\|\vec\lambda\right\|d_{match}(\rho_{(X,F^{\vec\varphi}_{\left(\vec\lambda,\vec\beta\right)}),k},\rho_{(Y,F^{\vec\psi}_{\left(\vec\lambda,\vec\beta\right)}),k})=d_{\left(\vec
\lambda,\vec\beta\right)}(\rho_{(X,\vec\p),k},\rho_{(Y,\vec\psi),k}),\nonumber\end{eqnarray*}}
and thus implying our claim.
\end{proof}

Roughly speaking, Proposition \ref{Insight3} implies that the bijection taking each pair $\left(\vec\lambda,\vec\beta\right)\in\Lambda\times B$ to $\left(\frac{\vec\lambda}{\|\vec\lambda\|},\vec\beta\right)$ deforms the set $\Lambda\times B$ into the product given by $\Lambda^*\times B$, in a way that does not affect the computation of the matching distance between multidimensional rank invariants. Therefore, in order to prove our main result, it is sufficient to define a correspondence between $\Lambda^*\times B$ and $Adm_n$, preserving the matching distance between multidimensional rank invariants. 
  
\begin{prop}\label{Insight4}
Let $\Lambda^*$ be the set containing all and only the vectors $\vec\lambda^*=(\lambda_1^*,\dots,\lambda_n^*)\in\R^n$ with $\lambda_i^*>0$ for every $i=1,\dots,n$, and $\left\|\vec\lambda^*\right\|=1$. Then it holds that $\underset{\left(\vec\lambda^*,\vec\beta\right)\in\Lambda^*\times B}{\sup} d_{\left(\vec\lambda^*,\vec\beta\right)}(\rho_{(X,\vec\varphi),k},\rho_{(Y,\vec\psi),k})=\underset{\left(\vec l,\vec b\right)\in Adm_n}{\sup} d_{\left(\vec l,\vec b\right)}(\rho_{(X,\vec\varphi),k},\rho_{(Y,\vec\psi),k})$.\end{prop}
\begin{proof}
Let us start by defining $f:\Lambda^*\times B\to Adm_n$ as the function taking each $\left(\vec\lambda^*,\vec\beta\right)\in\Lambda^*\times B$, with $\vec\lambda^*=\left(\lambda^*_1,\dots\lambda^*_n\right)$ and $\vec\beta=\left(\beta_1,\dots,\beta_n\right)$, to the pair $\left(\vec\lambda^*,\vec\beta-\frac{\sum_{i=1}^{n}\beta_i}{\sum_{i=1}^{n}\lambda^*_i}\vec\lambda^*\right)$. Such a function is well defined since  $\lambda_i^*>0$ for every index $i=1,\dots,n$. Moreover, $\im f\subseteq Adm_n$ since $\left\|\vec\lambda^*\right\|=1$ and $\sum_{j=1}^n\left(\beta_j-\frac{\sum_{i=1}^{n}\beta_i}{\sum_{i=1}^{n}\lambda^*_i}\lambda^*_j\right)=0$. 

We now prove that $f$ is actually surjective, that is, $\im f=Adm_n$. To this aim, let us fix $\left(\vec l,\vec b\right)\in Adm_n$ and consider a point $\left(\vec u,\vec v\right)\in\Delta^+$ such that $\left(\vec u,\vec v\right)\in\pi_{\left(\vec l,\vec b\right)}$. It is trivial to check that the half-planes collection $\{\pi_{\left(\vec\lambda^*,\vec\beta\right)}:\vec u=s\vec\lambda^*+\vec\beta,\ \vec v=t\vec\lambda^*+\vec\beta\ |\  s<t\}_{\left(\vec\lambda^*,\vec\beta\right)\in\Lambda^*\times B}$ satisfies a property analogous to the one described in Remark \ref{FoliationRem}, and hence there exist two real values $\hat s,\hat t$, with $\hat s<\hat t$, and a pair $\left(\vec\lambda^*,\vec\beta\right)\in\Lambda^*\times B$ such that 
\begin{eqnarray}\label{Eq2}
\left\{%
\begin{array}{ll}
    \vec u=\hat s\vec\lambda^* + \vec\beta\\
    \vec v=\hat t\vec\lambda^* + \vec\beta\\
\end{array}%
\right..
\end{eqnarray} 
By considering the change of coordinates given by $\sigma=s+\frac{\sum_{i=1}^n\beta_i}{\sum_{i=1}^n\lambda_i^*},\tau=t+\frac{\sum_{i=1}^n\beta_i}{\sum_{i=1}^n\lambda_i^*}$ (such a translation is well defined since  $\lambda_i^*>0$ for every index $i=1,\dots,n$), in equations (\ref{Eq2}) we obtain
\begin{eqnarray}\label{Eq3}
\left\{%
\begin{array}{ll}
    \vec u=\hat s\vec\lambda^* +\vec\beta=\left(\hat \sigma-\frac{\sum_{i=1}^n\beta_i}{\sum_{i=1}^n\lambda_i^*}\right)\vec\lambda^*+\vec\beta=\hat\sigma\vec\lambda^*+\vec b^*\\
    \vec v=\hat t\vec\lambda^* +\vec\beta=\left(\hat 
  \tau-\frac{\sum_{i=1}^n\beta_i}{\sum_{i=1}^n\lambda_i^*}\right)\vec\lambda^*+\vec\beta=\hat\tau\vec\lambda^*+\vec b^*\\
\end{array}%
\right.,
\end{eqnarray} 
where $\vec b^*=(b^*_1,\dots,b^*_n)=\left(\beta_1-\frac{\sum_{i=1}^n\beta_i}{\sum_{i=1}^n\lambda_i^*}\lambda_1^*,\dots,\beta_n-\frac{\sum_{i=1}^n\beta_i}{\sum_{i=1}^n\lambda_i^*}\lambda_n^*\right)$ and $\sum_{j=1}^{n} b_j^*=\sum_{j=1}^{n}\left(\beta_j-\frac{\sum_{i=1}^n\beta_i}{\sum_{i=1}^n\lambda_i^*}\lambda_j^*\right)=0$.
Therefore, $\left(\vec\lambda^*,\vec b^*\right)\in Adm_n$, thus implying (once more by Remark \ref{FoliationRem}) that $\left(\vec\lambda^*,\vec b^*\right)=\left(\vec l,\vec b\right)$. Given that  $f\left(\left(\vec\lambda^*,\vec\beta\right)\right)=\left(\vec\lambda^*,\vec\beta\right.-\left.\frac{\sum_{i=1}^n\beta_i}{\sum_{i=1}^n\lambda_i^*}\vec\lambda^*\right)=\left(\vec\lambda^*,\vec b^*\right)=\left(\vec l,\vec b\right)$, $f$ is surjective.

To conclude the proof, we observe that, from $\left(\vec l,\vec b\right)=f\left(\left(\vec\lambda^*,\vec\beta\right)\right)$, it follows that $F^{\vec\varphi}_{\left(\vec l,\vec b\right)}(x)=F^{\vec\varphi}_{\left(\vec\lambda^*,\vec\beta\right)}(x)+\frac{\sum_{i=1}^n\beta_i}{\sum_{i=1}^n\lambda_i}$ and $F^{\vec\psi}_{\left(\vec l,\vec b\right)}(y)=F^{\vec\psi}_{\left(\vec\lambda^*,\vec\beta\right)}(y)+\frac{\sum_{i=1}^n\beta_i}{\sum_{i=1}^n\lambda_i}$ for every $x\in X$ and for every $y\in Y$, respectively. Therefore, the claim easily follows from Proposition \ref{Prop3}, from the definition of the matching distance between $1$-dimensional rank invariants and from the surjectivity of $f$.
\end{proof}

We can now state our main result, claiming that any other distance between multidimensional rank invariants we can obtain by considering a set $\Lambda\times B$ different from $Adm_n$, coincides with the multidimensional matching distance $D_{match}$ introduced in Definition \ref{MultiDimDist}. It can be easily derived from Proposition \ref{Insight3} and Proposition \ref{Insight4}.

\begin{theorem}\label{InsightTotal}
Let $X$, $Y$ be two triangulable spaces, and $\vec\varphi:X\to\R^n$, $\vec\psi:Y\to\R^n$ two filtering functions. Consider also a set $\Lambda\times B\in\R^n\times\R^n$ and the half-planes collection $\left\{\pi_{\left(\vec\lambda,\vec\beta\right)}:\vec u=s\vec\lambda+\vec\beta,\ \vec v=t\vec\lambda+\vec\beta\ |\  s<t\right\}_{\left(\vec\lambda,\vec\beta\right)\in\Lambda\times B}$, assuming that the following consitions hold:
\begin{itemize}
\item for every $\left(\vec u, \vec v\right)\in\Delta^+$, there exists one and only one $\left(\vec\lambda,\vec\beta\right)\in\Lambda\times B$ such that $\left(\vec u, \vec v\right)\in\pi_{\left(\vec\lambda,\vec\beta\right)}$;
\item for every $\left(\vec\lambda,\vec\beta\right)\in\Lambda\times B$, the half-plane $\pi_{\left(\vec\lambda,\vec\beta\right)}$ is a subset of $\Delta^+$.
\end{itemize}
Then, setting $F_{\left(\vec\lambda,\vec\beta\right)}^{\vec\varphi}=\max_{i=1,\dots,n}\left\{\frac{\varphi_i-\beta_i}{\lambda_i}\right\}$ and $F_{\left(\vec\lambda,\vec
\beta\right)}^{\vec\psi}=\max_{i=1,\dots,n}\left\{\frac{\psi_i-\beta_i}{\lambda_i}\right\}$, it holds that the distance $\widehat D_{match}$ between multidimensional rank invariants, defined as 
$\widehat D_{match}(\rho_{(X,\vec\varphi),k},\rho_{(Y,\vec\psi),k})=\sup_{\left(\vec\lambda,\vec\beta\right)\in\Lambda\times B}\min_i\lambda_i\cdot d_{match}(\rho_{(X,F_{\left(\vec\lambda,\vec\beta\right)}^{\fr}),k},\rho_{(Y,F_{\left(\vec\lambda,\vec\beta\right)}^{\vec\psi}),k})$,
coincides with the multidimensional matching distance $D_{match}$.
\end{theorem}  

\section{Conclusion}
In this paper, we have shown that the foliation method allows us to define an infinite number of distances between multidimensional rank invariants. Indeed, by considering the half-planes collection given by $$\left\{\pi_{\left(\vec\lambda,\vec\beta\right)}:\vec u=s\vec\lambda+\vec\beta,\ \vec v=t\vec\lambda+\vec\beta\ |\  s<t\right\}_{\left(\vec\lambda,\vec\beta\right)\in\Lambda\times B},$$ with  $\Lambda\times B$ satisfying suitable hypotheses, we can introduce the distance $$\widehat D_{match}(\rho_{(X,\vec\varphi),k},\rho_{(Y,\vec\psi),k})=\sup_{\left(\vec\lambda,\vec\beta\right)\in\Lambda\times B}\min_i\lambda_i\cdot d_{match}(\rho_{(X,F_{\left(\vec\lambda,\vec\beta\right)}^{\fr}),k},\rho_{(Y,F_{\left(\vec\lambda,\vec\beta\right)}^{\vec\psi}),k})$$ between the rank invariants $\rho_{(X,\vec\varphi),k}$ and $\rho_{(Y,\vec\psi),k}$. The main result of the present work is the proof that all these distances coincide.
As an immediate consequence, our result opens the way to new procedures for the computation of the matching distance between multidimensional rank invariants. 

\bibliographystyle{abbrv}
\bibliography{CerriFrosiniRefs}

\end{document}